\documentclass[11pt]{article}

\usepackage[T1]{fontenc}
\usepackage[utf8]{inputenc}
\usepackage{amsmath,amssymb,amsthm}
\usepackage{enumitem}
\usepackage{microtype}
\usepackage{tikz}
\usepackage{graphicx}
\usetikzlibrary{arrows.meta,positioning,calc}
\usepackage[a4paper,left=26mm,right=26mm,top=25mm,bottom=27mm,headheight=14pt,headsep=9mm]{geometry}
\usepackage{titlesec}
\usepackage{fancyhdr}
\usepackage{caption}
\usepackage{etoolbox}

\titleformat{\section}{\large\bfseries}{\thesection.}{.6em}{}
\titleformat{\subsection}{\normalsize\bfseries}{\thesubsection.}{.6em}{}
\titleformat{\paragraph}[runin]{\normalsize\bfseries}{}{0pt}{}[.]
\titlespacing*{\section}{0pt}{2.2ex plus .6ex minus .3ex}{1ex plus .2ex}
\titlespacing*{\subsection}{0pt}{1.8ex plus .4ex minus .2ex}{.7ex plus .2ex}
\titlespacing*{\paragraph}{0pt}{1.4ex plus .3ex}{.7em}

\fancypagestyle{plain}{\fancyhf{}\fancyfoot[C]{\footnotesize\thepage}}

\renewenvironment{abstract}
  {\par\vspace{.6em}\begingroup\small
   \begin{list}{}{\setlength{\leftmargin}{6mm}\setlength{\rightmargin}{6mm}}
   \item[]\textbf{Abstract.}\enspace\ignorespaces}
  {\end{list}\endgroup\vspace{.3em}}
\newenvironment{keywords}
  {\par\begingroup\small\noindent\textbf{Keywords.}\enspace\ignorespaces}
  {\par\endgroup\vspace{.25em}}
\newenvironment{MSCcodes}
  {\par\begingroup\small\noindent\textbf{Mathematics Subject Classification(2020) }\enspace\ignorespaces}
  {\par\endgroup\vspace{1em}}
\AtBeginEnvironment{thebibliography}{\small}

\makeatletter
\renewcommand{\@maketitle}{%
  \newpage\null\vskip .3em
  \begin{center}%
    {\LARGE\bfseries\@title\par}%
    \vskip 1.3em
    {\normalsize\lineskip .5em
      \begin{tabular}[t]{c}\@author\end{tabular}\par}%

  \end{center}%
  \par\vskip .4em}
\makeatother
\usepackage{hyperref}
\usepackage[nameinlink,noabbrev]{cleveref}

\theoremstyle{plain}
\newtheorem{theorem}{Theorem}[section]
\newtheorem{lemma}[theorem]{Lemma}
\newtheorem{proposition}[theorem]{Proposition}
\newtheorem{corollary}[theorem]{Corollary}
\theoremstyle{definition}
\newtheorem{definition}[theorem]{Definition}
\theoremstyle{plain}

\setlist{nosep}
\allowdisplaybreaks
\newtheorem{remark}{Remark}

\tikzset{
  graphvertex/.style={circle,draw,fill=white,inner sep=1.1pt,
    minimum size=5.5mm,font=\small},
  grapharc/.style={-{Stealth[length=1.8mm,width=1.2mm]},line width=.48pt},
  forestarc/.style={-{Stealth[length=1.8mm,width=1.2mm]},line width=.65pt},
  witnessarc/.style={-{Stealth[length=2mm,width=1.35mm]},line width=1.15pt},
  levelguide/.style={densely dotted,gray},
  patternbox/.style={rounded corners,draw,align=center,inner sep=3pt,font=\small}
}

\newcommand{\adn}{\vec{\chi}_{a}}
\newcommand{\dn}{\vec{\chi}}
\newcommand{\dnk}{\vec{\chi}_{k}}
\newcommand{\adnk}{\vec{\chi}_{a,k}}
\newcommand{\TT}{\mathrm{TT}}
\newcommand{\Tstar}{T_5^{\ast}}
\newcommand{\Tn}{\mathcal{T}_n}
\newcommand{\Nplus}{N^{+}}
\newcommand{\Nminus}{N^{-}}

\hypersetup{
  hidelinks,
  pdftitle={Acyclic Dicolourings of Oriented Graphs: Paths, Random Tournaments, and Critical Orders},
  pdfauthor={Yihang Liu, Zhenyu Yang, Yuwan Zhang},
  pdfsubject={Acyclic dichromatic number of oriented graphs and tournaments},
  pdfkeywords={acyclic dichromatic number, oriented graph, path partition, Linial conjecture, tournament, directed path, random tournament, critical order}
}

\title{Acyclic Dicolourings of Oriented Graphs:\\
Paths, Random Tournaments, and Critical Orders}

\author{%
Yihang Liu\thanks{%
  School of Mathematical Sciences, Nankai University, Tianjin, China.\newline\hspace*{1.8em}
  \textit{E-mail:} \href{mailto:liuyihang@mail.nankai.edu.cn}{\nolinkurl{liuyihang@mail.nankai.edu.cn}}}%
\and
Zhenyu Yang\thanks{%
  School of Statistics and Data Science, LPMC and KLMDASR,
  Nankai University, Tianjin, China.\newline\hspace*{1.8em}
  \textit{E-mail:} \href{mailto:yangzhenyu@mail.nankai.edu.cn}{\nolinkurl{yangzhenyu@mail.nankai.edu.cn}}}%
\and
Yuwan Zhang\thanks{%
  Center for Combinatorics, Nankai University, Tianjin, China.\newline\hspace*{1.8em}
  \textit{E-mail:} \href{mailto:yuwanzhang2000@163.com}{\nolinkurl{yuwanzhang2000@163.com}}
  (corresponding author).}%
}
\begin{document}
\nocite{Champions,BangJensenGutin2009,BJPY,BergeKoptimal,PSL,Dilworth,Gallai,GallaiMilgram,GreeneKleitman1976,Grunbaum1973,Hasse,Linial,NL1982,Redei,Roy,Stanley,Vitaver}
\maketitle


\begin{abstract}
An acyclic dicolouring of an oriented graph is a vertex partition in which every colour class and every bipartite subdigraph induced by two classes is acyclic.
We first prove the Gallai--Roy-type bound $\adn(D)\leq L(D)$, where $L(D)$ is the maximum order of a directed path.
Let $r=\log_2(8/7)$.
Bang-Jensen, Picasarri-Arrieta, and Yeo previously constructed tournaments of order $n$ whose acyclic dichromatic number is at least $n-(\frac8r)\log_2 n-\log_2\log_2 n$.
We improve the leading logarithmic coefficient by a factor of two: for the uniform random tournament $\Tn$, asymptotically almost surely, $\adn(\Tn)\geq  n-\frac4r\log_2 n+\frac2r\log_2\log_2 n-O(1).$
Finally, if $m(k)$ denotes the minimum order of an oriented graph with acyclic dichromatic number at least $k$, tournament completion shows that the same minimum is obtained over tournaments.
We prove $m(3)=5$ and $m(4)=7$ and classify the tournament witnesses at these minimum orders: there is one at order five and two at order seven.
\end{abstract}

\begin{MSCcodes}
05C20, 05C15, 05C80
\end{MSCcodes}
\begin{keywords}
acyclic dicolouring, path partition, random tournament,
first-moment method, Gallai--Roy theorem, critical order
\end{keywords}

\section{Introduction}\label{sec:introduction}

\paragraph{From posets to directed graphs}
A natural starting point for the path--colouring framework is the theory of partially ordered sets.
Dilworth's theorem identifies the minimum number of chains partitioning a finite poset with its maximum antichain size~\cite{Dilworth}.
Greene and Kleitman extended this min--max relation to every positive
integer $k$: the minimum $k$-norm
$\sum_{C\in\mathcal P}\min\{|C|,k\}$ of a chain partition equals the
maximum number of elements covered by $k$ antichains~\cite{GreeneKleitman1976}.
The dual identity, interchanging chains and antichains, was proved by
Greene~\cite{Greene1976}.
Representing a poset by its transitive acyclic digraph turns chains into directed paths and antichains into stable sets.
Gallai--Milgram~\cite{GallaiMilgram} and Gallai--Hasse--Roy--Vitaver~\cite{Gallai,Hasse,Roy,Vitaver} show that the two $k=1$ inequalities survive for arbitrary digraphs.
Linial conjectured the full extensions $\pi_k(D)\leq\alpha_k(D),$ and $ \chi_k(D)\leq\lambda_k(D),$ where the left sides are minimum $k$-norms of directed-path partitions and proper colourings, while the right sides are maximum coverage by at most $k$ stable sets and at most $k$ vertex-disjoint directed paths, respectively~\cite{Linial}.
Berge's $k$-optimal path-partition conjecture strengthens the first inequality by requiring a partial $k$-colouring that meets each path $P$ in $\min\{|V(P)|,k\}$ distinct stable colour classes~\cite{BergeKoptimal}.
This poset-to-digraph programme supplies the natural background for asking which path--colouring relations persist when stable sets are replaced by acyclic sets.

\paragraph{From dicolouring to acyclic dicolouring}
Neumann-Lara made this replacement through the \emph{dichromatic number} $\dn(D)$, the minimum number of vertex classes inducing acyclic subdigraphs~\cite{NL1982}.
If a dicolouring $\mathcal S$ is assigned the $k$-norm $\sum_{S\in\mathcal S}\min\{|S|,k\}$, let $\dnk(D)$ denote the minimum such value.
De Paula Silva, Nunes da Silva, and Lee recently proved that every minimum dicolouring has a directed path meeting each colour class exactly once and, more generally, that $\dnk(D)\leq\lambda_k(D)$~\cite{PSL}.
Thus induced acyclic sets recover important parts of the classical path duality.

Ordinary dicolouring, however, controls cycles only inside individual colour classes: a directed cycle may still alternate between two colours.
To control this interaction, and motivated by Gr\"unbaum's acyclic colouring of undirected graphs~\cite{Grunbaum1973},   Bang-Jensen, Picasarri-Arrieta, and Yeo introduced the \emph{acyclic dichromatic number} $\adn(D)$ of an oriented graph~\cite{BJPY}.
Here every colour class and every bipartite subdigraph induced by two classes must be acyclic.
For comparison with the Linial-type parameters, we define $\adnk(D)$ to be the corresponding minimum $k$-norm.
The extra condition is substantial: $\adn(D)-\dn(D)$ is unbounded even among tournaments with dichromatic number two.
Initial work developed critical and extremal questions for the parameter, and subsequent work connected it with tournament heroes and local-to-global structure~\cite{BJPY,Champions}.
In tournaments the two obstructions are concrete---a directed triangle within one class and an alternating directed $4$-cycle between two classes---making the class a natural test ground for path bounds, finite critical structure, and asymptotic extremality.
The present paper addresses these three questions in the order path theory, probabilistic extremality, and critical orders.

\paragraph{A Gallai--Roy-type result}
The preceding path result gives $\dn(D)\leq L(D)$, but it does not imply the corresponding statement for $\adn$: alternating bichromatic cycles are invisible to the individual colour classes.
Our first result shows that the Gallai--Roy numerical bound nevertheless survives unchanged.

\begin{theorem}\label{thm:intro-path}
For every finite oriented graph $D$,
$$
 \adn(D)\leq L(D).
$$
\end{theorem}

The proof uses the level sets of a forward depth-first-search forest.
A separation property forces every $D[C_i]$ and every bipartite subdigraph $D[C_i,C_j]$, with $i\neq j$, to be acyclic, while a deepest root-to-vertex tree path certifies that the number of levels is at most $L(D)$.
The oriented hypothesis is essential because a digon defeats both parts of the colouring definition.

\paragraph{Random tournaments}
The path bound can be close to equality.
Let $\Tn$ be the uniform random tournament on $n$ labelled vertices, and put $ r=\log_2(8/7), a=\frac4r, \beta=\frac2r. $

Bang-Jensen, Picasarri-Arrieta, and Yeo proved that, for every sufficiently large $n$, there exists a tournament $T$ of order $n$ such that
$$
 \adn(T)\geq
 n-\frac8r\log_2 n-\log_2\log_2 n.
$$
Their probabilistic proof uses Stanley's enumeration of acyclic orientations, including the exact values for $K_{2,2}$, $K_{2,3}$, and $K_{3,3}$, and analyses partitions built from blocks of orders two and three~\cite{BJPY}.

We retain these local enumerative inputs but replace the restricted block analysis by a complete profile count.
An acyclic dicolouring with $n-s$ colours saves $s$ colours relative to the all-singleton colouring.
We count every possible profile of its nonsingleton colour classes and prove that, throughout the logarithmic regime, profiles consisting entirely of pairs asymptotically dominate the resulting unweighted first moment.
This changes the leading coefficient from $\frac8{\log_2(8/7)}\text{to} \frac4{\log_2(8/7)}$ and also determines the second-order term produced by this counting scheme.
\begin{theorem}\label{thm:intro-random}
Asymptotically almost surely,
$$
 \adn(\Tn)\geq n-a\log_2n+\beta\log_2\log_2n-O(1).
$$
\end{theorem}

The two coefficients are the boundary coefficients for the unweighted first moment that counts every admissible block system once.
This statement concerns this particular union-bound scheme; it does not preclude weighted first-moment arguments, structural pruning, or methods using dependencies between block systems.

\paragraph{Critical orders}
The random result concerns large examples; at the opposite scale, one may ask for the first obstruction to a given number of colours.
For $k\geq1$, define
$$
 m(k)=\min\{|V(D)|:D\text{ is a finite oriented graph and }\adn(D)\geq k\}.
$$
Every oriented graph can be completed to a tournament without decreasing $\adn$, so the numerical value of $m(k)$ is unchanged when the minimum is restricted to tournaments.
Classification statements, however, are made only for tournament witnesses.

\begin{theorem}\label{thm:intro-critical}
One has $m(3)=5$ and $m(4)=7$.
Up to isomorphism, the unique tournament witness of order five is $\Tstar$, and the tournament witnesses of order seven are precisely $P_7$ and $Q_7$.
These tournaments have acyclic dichromatic numbers three, four, and four, respectively.
\end{theorem}

The order-five case follows from score sequences.
At order seven, the size of a largest transitive subtournament splits the proof into a regular case, which forces $P_7$, and a transitive-four-set case, which produces $Q_7$ through an inclusion-antichain analysis.
The main text develops the comparison idea, and Appendix~A contains the complete case analysis.

\section{Terminology and elementary facts}\label{sec:terminology}

All digraphs considered in this paper are finite and nonempty and have neither loops nor parallel arcs.
For terminology not defined here, we follow \cite{BangJensenGutin2009}.
An \emph{oriented graph} is a digraph with no pair of opposite arcs, and a \emph{tournament} is an oriented graph in which exactly one of $uv$ and $vu$ is an arc for every pair of distinct vertices $u,v$.

We write $u\to v$ when $uv\in A(D)$.
The out- and in-neighbourhoods of $v$ are denoted by $\Nplus_D(v)$ and $\Nminus_D(v)$, respectively, and the corresponding degrees by $d_D^+(v)$ and $d_D^-(v)$.
Subscripts are omitted when the digraph is clear.
For disjoint sets $X,Y\subseteq V(D)$, the bipartite subdigraph $D[X,Y]$ contains all arcs with one end in each set.
We write $X\Rightarrow Y$ when every vertex of $X$ dominates every vertex of $Y$.

A directed path is a sequence $(v_1,\ldots,v_p)$ of distinct vertices with $v_i\to v_{i+1}$ for $1\leq i<p$; its order is $p$.
Let $L(D)$ be the maximum order of a directed path in $D$.
A digraph is \emph{acyclic} if it has no directed cycle.
A topological ordering lists the vertices of an acyclic digraph so that every arc points forward.
A displayed chain such as $x\to y\to z$ asserts only the two consecutive arcs unless an additional arc is stated.

\begin{definition}\label{def:colouring}
For a positive integer $q$, write $[q]=\{1,\ldots,q\}$.%
A colouring with nonempty classes $C_1,\ldots,C_q$ is a \emph{$q$-dicolouring} if every $D[C_i]$ is acyclic; its minimum possible number of classes is $\dn(D)$.
If $D$ is oriented, the colouring is an \emph{acyclic $q$-dicolouring} if it is a dicolouring and $D[C_i,C_j]$ is acyclic whenever $i\neq j$.
Its minimum possible number of classes is $\adn(D)$.
\end{definition}

The assumption that \(D\) is oriented is essential.
If both \(u\to v\) and \(v\to u\) are present, then \(u\) and \(v\) cannot belong either to the same colour class or to two distinct colour classes: in the former case they form a directed \(2\)-cycle within one class, and in the latter they form a directed \(2\)-cycle between two classes.
On the other hand, every oriented graph admits an acyclic dicolouring by assigning a distinct colour to each vertex.
We refer to this as the \emph{all-singleton colouring}.

A tournament is acyclic if and only if it is transitive.
We denote the transitive tournament of order $t$ by $\TT_t$.
Its \emph{transitive order} lists all vertices so that every arc points from left to right.

For a positive integer $k$ and a dicolouring $\mathcal S$, define its $k$-norm by $\|\mathcal S\|_k=\sum_{S\in\mathcal S}\min\{|S|,k\}.$
Let $\dnk(D)$ be the minimum $k$-norm over all dicolourings, and let $\adnk(D)$ be the corresponding minimum over acyclic dicolourings of an oriented graph.
In particular, $\dn_1(D)=\dn(D)$ and $\vec{\chi}_{a,1}(D)=\adn(D)$.

\begin{lemma}\label{lem:local}
Let $T$ be a tournament.
\begin{enumerate}[label=\textnormal{(\alph*)}]
\item A subtournament of $T$ is acyclic if and only if it contains no directed triangle.
\item If $X$ and $Y$ are disjoint vertex sets, then $T[X,Y]$ is acyclic if and only if it contains no directed $4$-cycle.
Every such cycle alternates between $X$ and $Y$.
\end{enumerate}
\end{lemma}

\begin{proof}
Choose a shortest directed cycle.
In a tournament, any cycle of length at least four has a chord whose orientation produces a shorter directed cycle, proving (a).
In a bipartite orientation, every directed cycle alternates between the two parts.
A shortest one has length at least four; if its length is at least six, the arc joining its first and fourth vertices produces either a directed $4$-cycle or a shorter directed cycle.
This proves (b).
\end{proof}

\begin{lemma}\label{lem:refinement}
Suppose that $\mathcal C$ is an acyclic dicolouring of an oriented graph $D$.
If one or more colour classes of $\mathcal C$ are partitioned into smaller nonempty classes, then the resulting colouring is again an acyclic dicolouring.
\end{lemma}

\begin{proof}
Let $A$ and $B$ be two classes of the new colouring.
Each new class is contained in some old colour class.
Thus $D[A]$ is a subdigraph of an old colour class and is acyclic.

If $A$ and $B$ come from the same old class $C_i$, then $D[A,B]$ is a subdigraph of $D[C_i]$, and hence is acyclic.
If they come from two distinct old classes $C_i$ and $C_j$, then $D[A,B]$ is a subdigraph of $D[C_i,C_j]$, which is also acyclic.
\end{proof}

\section{A Gallai--Roy-type path bound}\label{sec:paths}

\subsection{The Gallai--Hasse--Roy--Vitaver principle}
Let $U(D)$ be the underlying simple graph of a digraph $D$.
A collection $\mathcal H$ of pairwise disjoint vertex sets and a subdigraph $Q$ are \emph{orthogonal} if every member of $\mathcal H$ contains exactly one vertex of $Q$.

\begin{theorem}[Gallai--Hasse--Roy--Vitaver~\cite{Gallai,Hasse,Roy,Vitaver}]\label{thm:GHRV}
Let $D$ be a digraph and let $P$ be a prescribed longest directed path of $D$.
There is a proper colouring of $U(D)$ whose colour classes are orthogonal to $P$.
In particular, $\chi(U(D))\leq |V(P)|=L(D).$
\end{theorem}

De Paula Silva, Nunes da Silva, and Lee proved the reverse statement for ordinary dicolourings~\cite{PSL}.

\begin{theorem}[De Paula Silva--Nunes da Silva--Lee~\cite{PSL}]\label{thm:reverse-orth}
Let $D$ be a digraph and let $\mathcal S=\{S_1,\ldots,S_t\}$ be a minimum dicolouring of $D$.
Then $D$ has a directed path containing exactly one vertex of each $S_i$.
\end{theorem}

Their argument also gives $\dnk(D)\leq\lambda_k(D)$, where $\lambda_k(D)$ is the maximum number of vertices covered by at most $k$ pairwise vertex-disjoint directed paths.
The next result addresses the stronger two-class condition in Definition~\ref{def:colouring}.

\subsection{A rooted-forest decomposition}
A rooted out-forest $F$ has every tree arc directed from parent to child.
Write $F_x$ for the subtree rooted at $x$, and let $d_F(v)$ be the number of vertices on the root-to-$v$ tree path.

\begin{lemma}\label{lem:level-decomposition}
Every finite digraph $D$ has a spanning rooted out-forest $F$ and a linear order $\tau$ of its vertices with the following separation property:
\begin{equation}
 u\in V(F_x),\quad u\to v,\quad v\notin V(F_x)
 \quad\Longrightarrow\quad \tau(v)<\tau(x).\label{eq:separation}
\end{equation}
If $D$ is oriented, the nonempty level sets $C_i=\{v\in V(D):d_F(v)=i\}$ form an acyclic dicolouring of $D$.

\end{lemma}

\begin{proof}
Take a standard forward depth-first spanning forest and let $\tau$ be its discovery order.
If the hypotheses in \eqref{eq:separation} hold, then $v$ cannot be undiscovered while the subtree rooted at $x$ is being explored, for otherwise it would belong to that subtree.
Hence $v$ precedes $x$, proving the separation property.

If $u\to v$ with $u,v\in C_i$, neither vertex is a descendant of the other, so \eqref{eq:separation} gives $\tau(v)<\tau(u)$.
Thus decreasing $\tau$ is a topological order of $D[C_i]$.

Now let $i<j$ and suppose that $D[C_i,C_j]$ contains a directed cycle.
Choose on it a vertex $x$ of minimum $\tau$.
If $x\in C_j$, its successor lies outside $F_x$, contradicting \eqref{eq:separation}; hence $x\in C_i$.
Let $x\to y\to z$ be the first two arcs of the cycle.
The minimality of $\tau(x)$ and \eqref{eq:separation} imply that $y$ is a descendant of $x$.
Because $D$ is oriented, the cycle has length at least four, so $z\neq x$.
Since $z$ and $x$ have the same depth, $z\notin V(F_x)$, and applying \eqref{eq:separation} to $y\to z$ gives $\tau(z)<\tau(x)$, a contradiction.
Therefore every $D[C_i,C_j]$ is acyclic.
\end{proof}

The level decomposition and the path certifying its number of nonempty levels are illustrated in Figure~\ref{fig:level-forest}.

\begin{figure}[t]
\centering
\includegraphics{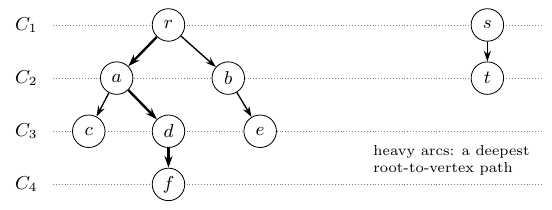}
\caption{The level decomposition in Lemma~\ref{lem:level-decomposition}.
The solid arrows are the tree arcs of the particular rooted out-forest $F$ supplied by the lemma; all other arcs of $D$ are suppressed.
The horizontal levels are the colour classes $C_i$.
The heavy path meets every nonempty level once and certifies that the number of levels is at most $L(D)$.}
\label{fig:level-forest}
\end{figure}
\begin{theorem}\label{thm:path-bound}
Every finite oriented graph $D$ satisfies
$$
 \adn(D)\leq L(D).
$$
\end{theorem}

\begin{proof}
Let $F$ be supplied by Lemma~\ref{lem:level-decomposition}, and let $h$ be its maximum depth.
Its nonempty levels form an acyclic dicolouring, so $\adn(D)\leq h$.
A root-to-vertex tree path of maximum depth is a directed path of order $h$, and hence $h\leq L(D)$.
\end{proof}

If directed-path length is counted by arcs and $\ell(D)$ denotes the maximum such length, then $L(D)=\ell(D)+1$ and Theorem~\ref{thm:path-bound} reads $\adn(D)\leq\ell(D)+1$.

\section{Random tournaments and the optimised first-moment bound}\label{sec:random}

The uniform random tournament $\Tn$ is obtained by orienting every edge of $K_n$ independently, with the two directions equally likely.
An event holds \emph{asymptotically almost surely}, abbreviated a.a.s., if its probability tends to one as $n\to\infty$.
Throughout this section, all logarithms are to base $2$ unless otherwise indicated.

If an acyclic dicolouring has $q$ classes, then it saves $n-q$ colours relative to the all-singleton colouring.
We call each non-singleton colour class a \emph{block}; a block of size $2$ will be called a \emph{pair-block}.
A block of size $b$ contributes $b-1$ to this saving.

Throughout this section, put $r=\log_2(8/7)$, $a=4/r$, and $\beta=2/r$.
Numerically, $r\approx0.192645$, $a\approx20.763572$, and $\beta\approx10.381786$.

\subsection{Acyclic orientations of complete bipartite graphs}
Let $A_{p,q}$ be the number of acyclic orientations of $K_{p,q}$, where $p,q\ge2$, and define the \emph{acyclicity defect} $\delta_{p,q}=pq-\log_2 A_{p,q}$.
A uniformly random orientation of $K_{p,q}$ is acyclic with probability $2^{-\delta_{p,q}}$.
Stanley's theorem gives $A_{p,q}=|P_{K_{p,q}}(-1)|$, where $P_G$ is the chromatic polynomial~\cite{Stanley}.
In particular, as already used in~\cite{BJPY}, $A_{2,2}=14$, $A_{2,3}=46$, and $A_{3,3}=230$, so $\delta_{2,2}=r$.

\begin{lemma}\label{lem:bipartite-defect}
For all integers $p,q\ge2$,
$$
 \delta_{p,q}\ge r(p-1)(q-1).
$$
Equality holds for $p=q=2$.
\end{lemma}

\begin{proof}
The statement is symmetric, so assume $2\le p\le q$.
The cases $(2,2),(2,3),(3,3)$ follow from the exact counts above: $\delta_{2,2}=r$, while $\delta_{2,3}\approx0.476438>2r$ and $\delta_{3,3}\approx1.154510>4r$.

Every acyclic orientation of $K_{p,q}$ has a topological ordering.
Read the $p$ vertices in the first part in the order in which they occur; there are at most $p!$
possibilities.
Each vertex of the second part lies in one of the $p+1$ gaps before, between, or after these vertices.
The order and gap choices determine every bipartite arc, and hence $A_{p,q}\le p!(p+1)^q$.
It is therefore enough to prove
$$
 F_p(q):=pq-r(p-1)(q-1)-\log_2(p!)-q\log_2(p+1)\ge0.
$$
For fixed $p$, this is affine in $q$ with slope $\sigma_p=p-r(p-1)-\log_2(p+1)$.
We have $\sigma_2>0$, and
$$
 \sigma_{p+1}-\sigma_p
 =1-r-\log_2\frac{p+2}{p+1}
 \ge1-r-\log_2(4/3)>0.
$$
Thus $F_p(q)$ increases with $q$.
Direct calculation gives $F_2(4)\approx0.082215$ and $F_3(4)\approx0.259167$.
For $p\ge4$, it remains to check $H(p):=F_p(p)$.
We have $H(4)\approx0.393519$, while
$$
 H(p+1)-H(p)
 =2p+1-r(2p-1)-2\log_2(p+1)
 -(p+1)\log_2\left(1+\frac1{p+1}\right).
$$
Using $\ln(1+x)\le x$, this is at least $2(1-r)p+(1+r)-2\log_2(p+1)-\log_2 e,$ which is positive at $p=4$ and strictly increasing for real $p\ge4$.
Hence $H(p)>0$ for every $p\ge4$.
\end{proof}

The factor $(p-1)(q-1)$ is measured in units of colour saving: a block of size $p$ saves $p-1$ colours.
Equality at $(2,2)$ is the first indication that pairs should be extremal.

\subsection{Block profiles and the exact first moment}
Fix a finite-support vector $\mathbf x=(x_2,x_3,\dots)$ of nonnegative integers.
A block system has \emph{profile} $\mathbf x$ if it contains exactly $x_b$ blocks of order $b$.
Set
$$
 s=\sum_{b\ge2}(b-1)x_b,
 \qquad M=\sum_{b\ge2}b x_b,
 \qquad z=\sum_{b\ge3}(b-2)x_b.
$$
Here $s$ is the colour saving, $M$ is the number of selected vertices, and $z$ is the excess over the all-pair-block profile.
We call the profile \emph{feasible at order $n$} when $M\leq n$; for an infeasible profile the associated count is zero.
The identity $M=2s-z$ follows from $b=2(b-1)-(b-2)$.

Let $N_{\mathbf x}$ count unordered families of disjoint vertex sets in $\Tn$ having profile $\mathbf x$ such that every block is transitive and every two blocks have an acyclic bipartite interaction.
Vertices outside the blocks receive private colours and impose no further condition.

\begin{proposition}\label{prop:profile-expectation}
For every feasible profile $\mathbf x$,
$$
 \mathbb E N_{\mathbf x}
 =\frac{n!}{(n-M)!\prod_{b\ge2}x_b!}\,2^{-K(\mathbf x)},
$$
where
$$
 K(\mathbf x)
 =\sum_{b\ge2}\binom b2x_b
 +\sum_{p<q}\delta_{p,q}x_px_q
 +\sum_{b\ge2}\delta_{b,b}\binom{x_b}{2}.
$$
\end{proposition}

\begin{proof}
There are $n!/((n-M)!\prod_{b\ge2}(b!)^{x_b}x_b!)$ unordered block families.
A fixed $b$-set is transitive with probability $b!/2^{\binom b2}$.
Two fixed blocks of orders $p$ and $q$ interact acyclically with probability $2^{-\delta_{p,q}}$.
These events concern disjoint sets of tournament edges and are therefore independent.
The factors $(b!)^{x_b}$ cancel, giving the formula.
\end{proof}

An \emph{acyclic matching} is a family of pairwise disjoint vertex pairs such that every two pairs have an acyclic bipartite interaction.
For $2s\leq n$, let $X_s$ count acyclic matchings with $s$ pairs.
This is the all-pair-block profile $x_2=s$, and Proposition~\ref{prop:profile-expectation} gives the exact formula
\begin{equation}\label{eq:matching-expectation}
 \mathbb E X_s
 =\frac{n!}{(n-2s)!s!}\,2^{-s-r\binom s2}.
\end{equation}

\subsection{Dominance of size-\texorpdfstring{$2$}{2} blocks}
\begin{proposition}\label{prop:pair-dominance}
Let $\mathbf x$ have saving $s$ and excess $z$, and assume $2s\leq n$.
Then
$$
 \frac{\mathbb E N_{\mathbf x}}{\mathbb E X_s}
 \le
 \left(\frac{s^2}{n-2s+1}\right)^z.
$$
Consequently, if $s=O(\log n)$ and $z\ge1$, then $\mathbb E N_{\mathbf x}=o(\mathbb E X_s)$ uniformly over all such profiles.
\end{proposition}

\begin{proof}
Index the non-singleton blocks by $B_1,\dots,B_t$, put $b_i=|B_i|$ and $d_i=b_i-1$, and note that $\sum_i d_i=s$.
By Lemma~\ref{lem:bipartite-defect},
\begin{align*}
 K(\mathbf x)
 &\ge \sum_i\left(d_i+\binom{d_i}{2}\right)+r\sum_{i<j}d_id_j\\
 &=s+r\binom s2+(1-r)\sum_i\binom{d_i}{2}\\
 &\ge s+r\binom s2.
\end{align*}
Thus the exponential factor in Proposition~\ref{prop:profile-expectation} is no larger than the all-pair-block exponential factor in \eqref{eq:matching-expectation}, and
$$
 \frac{\mathbb E N_{\mathbf x}}{\mathbb E X_s}
 \le
 \frac{(n-2s)!}{(n-M)!}\,
 \frac{s!}{\prod_{b\ge2}x_b!}.
$$
Since $M=2s-z$ and $2s\le n$, the first factor is the reciprocal of $z$ positive consecutive integers beginning with $n-2s+1$, and is therefore at most $(n-2s+1)^{-z}$.
Let $R=\sum_{b\ge3}x_b$.
Since every large block contributes at least one to $z$, we have $R\le z$, while $s=x_2+z+R$.
Hence $s-x_2\le2z$, and therefore $s!/\prod_{b\ge2}x_b!\le s!/x_2!\le s^{2z}$.
Multiplying the bounds proves the claim.
\end{proof}

Let $Y_s=\sum_{\mathbf x:\,\sum(b-1)x_b=s}N_{\mathbf x}$ count all block systems that can occur as the non-singleton classes of an acyclic dicolouring saving exactly $s$ colours.

\begin{corollary}\label{cor:pair-asymptotics}
For every fixed $C>0$, uniformly for $s\le C\log_2n$,
$$
 \mathbb EY_s=(1+o(1))\mathbb EX_s.
$$
\end{corollary}

\begin{proof}
The all-pair-block profile contributes $\mathbb EX_s$.
For fixed positive excess $z$, the entries $x_b$ with $b\ge3$ encode a partition of $z$ into parts $b-2$, so there are at most $p(z)\le2^z$ such profiles, where $p(z)$ is the integer partition function.
When $s\le C\log_2n$, for large $n$ we have $n-2s+1\ge n/2$.
By Proposition~\ref{prop:pair-dominance}, the total contribution relative to $\mathbb EX_s$ is at most
$$
 \sum_{z\ge1}2^z\left(\frac{s^2}{n-2s+1}\right)^z
 \le\sum_{z\ge1}\left(\frac{4s^2}{n}\right)^z=o(1).
$$
\end{proof}

Thus, throughout the logarithmic regime, pair-blocks supply a $1+o(1)$ fraction of the complete first moment.
This profile comparison is logically separate from the threshold calculation below.

\subsection{The first-moment threshold and the explicit random theorem}
Write $L_n=\log_2n$ and $\ell_n=\log_2L_n$.
For $k=O(\log n)$, Stirling's formula and \eqref{eq:matching-expectation} give
\begin{align}
 \log_2\mathbb EX_k
 &=2kL_n-\frac r2k^2-k\log_2k
 +\left(\log_2e-1+\frac r2\right)k-O(\log k).
 \label{eq:log-first-moment}
\end{align}
Set $k=aL_n-\beta\ell_n+C$, where $C$ remains in a bounded interval.
The expansion below is uniform for bounded $C$; in particular, uniformly for $C\in[0,1]$,
\begin{equation}\label{eq:threshold-linear}
 \log_2\mathbb EX_k
 =\left(2+a(\log_2e-1-\log_2a)-2C\right)L_n
 +O(\ell_n^2).
\end{equation}
The bracketed coefficient is approximately $-79.669104$ at $C=0$ and decreases as $C$ increases.
Hence it is uniformly negative for $C\in[0,1]$.

\begin{proposition}\label{prop:two-sided-threshold}
Let $c>0$ and $d\in\mathbb R$ be fixed, and let $s=cL_n-d\ell_n+O(1)$.
Then
\begin{equation}\label{eq:two-sided-threshold}
 \log_2\mathbb E X_s
 =\left(2c-\frac r2c^2\right)L_n^2
 +\bigl((-2+rc)d-c\bigr)L_n\ell_n+O(L_n).
\end{equation}
Consequently, the coefficient of $L_n^2$ is positive for $0<c<a$ and negative for $c>a$.
At $c=a$, the coefficient of $L_n\ell_n$ equals $2d-a$; it is positive for $d>\beta$ and negative for $d<\beta$.
\end{proposition}

\begin{proof}
Substitute $s=cL_n-d\ell_n+O(1)$ into \eqref{eq:log-first-moment}.
The terms $2sL_n-rs^2/2$ contribute $(2c-rc^2/2)L_n^2+(-2+rc)dL_n\ell_n+O(L_n)$, while $-s\log_2s$ contributes $-cL_n\ell_n+O(L_n)$.
All remaining terms are $O(L_n)$, which gives \eqref{eq:two-sided-threshold}.
The sign assertions follow from $a=4/r$ and $\beta=a/2=2/r$.
\end{proof}

Since $Y_s\ge X_s$, divergence of the pair-block contribution prevents this Markov/union-bound nonexistence argument from proving a stronger pair of threshold coefficients.
This is a limitation of the method: divergence of the first moment does not imply the existence of a colouring.

\begin{theorem}\label{thm:random-explicit}
Let $s_n=\lceil a\log_2n-\beta\log_2\log_2n\rceil$.
Then, asymptotically almost surely, $n-\adn(\Tn)\le s_n-1$.
\end{theorem}

\begin{proof}
Write $s_n=aL_n-\beta\ell_n+C_n$, where $0\le C_n<1$.
The uniform form of \eqref{eq:threshold-linear} gives $\mathbb EX_{s_n}=o(1)$, and Corollary~\ref{cor:pair-asymptotics} gives $\mathbb EY_{s_n}=o(1)$.

Suppose an acyclic dicolouring saves at least $s_n$ colours.
If it saves more, split one vertex at a time from a non-singleton class.
By Lemma~\ref{lem:refinement}, every intermediate partition remains an acyclic dicolouring, and every split lowers the saving by exactly one.
Hence there is an acyclic dicolouring saving exactly $s_n$ colours, so $Y_{s_n}\ge1$.
Markov's inequality now yields
$$
 \Pr(n-\adn(\Tn)\ge s_n)\le\Pr(Y_{s_n}\ge1)\le\mathbb EY_{s_n}=o(1).
$$
\end{proof}

\begin{remark}\label{rem:method-optimality}
Corollary~\ref{cor:pair-asymptotics} and Proposition~\ref{prop:two-sided-threshold} show that the coefficients $a$ and $\beta$ are optimal for the unweighted first moment of $Y_s$, equivalently for this particular union-bound scheme.
No optimality is asserted for weighted variants or for the true typical value of $\adn(\Tn)$.
\end{remark}

\subsection{A path consequence}
\begin{corollary}[Asymptotic sharpness of the path bound]\label{cor:path-sharp}
For the uniform random tournament $\Tn$, a.a.s.
$\frac{\adn(\Tn)}{L(\Tn)}\longrightarrow1.$
Consequently, there is no universal $c<1$ such that $\adn(D)\le cL(D)+O(1)$ for every oriented graph $D$.
\end{corollary}

\begin{proof}
By R\'edei's theorem~\cite{Redei}, every tournament has a directed Hamilton path, so $L(\Tn)=n$.
Apply \Cref{thm:random-explicit}.
\end{proof}

\section{Critical orders in oriented graphs}\label{sec:critical}

For $k\geq1$, define
\begin{equation}\label{eq:mk-definition}
 m(k)=\min\{|V(D)|:D\text{ is a finite oriented graph and }\adn(D)\geq k\}.
\end{equation}
An oriented graph $D$ is \emph{$k$-critical} if $\adn(D)=k$ and every proper induced subdigraph has acyclic dichromatic number at most $k-1$.

\begin{proposition}\label{prop:tournament-reduction}
The value of $m(k)$ is unchanged if the minimum in \eqref{eq:mk-definition} is restricted to tournaments.
Moreover, every minimum-order oriented graph $D$ with $\adn(D)\geq k$ is $k$-critical.
\end{proposition}

\begin{proof}
Complete an oriented graph $D$ to a tournament $T$ on the same vertex set by orienting every missing edge.
Every acyclic dicolouring of $T$ is also one of $D$, so $\adn(D)\leq\adn(T)$.
Since every tournament is an oriented graph, the two minima coincide.

If $D$ has minimum order, each proper induced subdigraph has acyclic dichromatic number at most $k-1$.
For $k\geq2$, colour $D-v$ with at most $k-1$ colours and give $v$ a private colour.
The bipartite oriented subdigraph between a singleton and another class contains no directed cycle, so this is an acyclic dicolouring using at most $k$ colours.
Hence $\adn(D)=k$; the case $k=1$ is immediate.
\end{proof}

Thus numerical questions about $m(k)$ reduce to tournaments, although the parameter itself is defined for all finite oriented graphs.
The classification statements below concern tournament witnesses only.

\subsection{General bounds}
We first invert the deterministic and probabilistic estimates from the preceding sections.

\begin{lemma}\label{lem:transitive-log}
Every tournament of order $n$ contains a transitive subtournament of order at least $\lfloor\log_2n\rfloor+1$.
Consequently, $\adn(T)\leq n-\lfloor\log_2n\rfloor$ for every tournament $T$ of order $n$.
\end{lemma}

\begin{proof}
The first assertion follows by induction: a vertex has at least half of the remaining vertices as out-neighbours or at least half as in-neighbours, and adjoining it to a transitive set in that neighbourhood preserves transitivity.
Give such a transitive set one colour and every remaining vertex a private colour.
\end{proof}

Bang-Jensen, Picasarri-Arrieta, and Yeo proved that every tournament $T$ of order $n$ satisfies~\cite{BJPY}
\begin{equation}\label{eq:BJPY-universal}
 \adn(T)\leq n-\alpha\log_2n+27,
 \qquad \alpha=\frac1{\log_2(4/3)}.
\end{equation}

\begin{theorem}\label{thm:m-bounds}
For every positive integer $k$,
$$
 m(k)\geq
 \max\left\{
 k+\lfloor\log_2k\rfloor,
 \left\lceil k+\alpha\log_2k-27\right\rceil
 \right\}.
$$
For all sufficiently large $k$,
$$
 m(k)\leq
 k+\left\lceil a\log_2k-\beta\log_2\log_2k\right\rceil.
$$
\end{theorem}

\begin{proof}
Let $D$ be an oriented graph of order $n$ with $\adn(D)\geq k$, and complete it to a tournament $T$.
Then $\adn(T)\geq k$.
Lemma~\ref{lem:transitive-log} gives $k\leq n-\lfloor\log_2n\rfloor \leq n-\lfloor\log_2k\rfloor,$ and \eqref{eq:BJPY-universal} gives $k\leq n-\alpha\log_2k+27$.
These are the two lower bounds.

For the upper bound, put $F(x)=a\log_2x-\beta\log_2\log_2x$ and $n=k+\lceil F(k)\rceil$.
Since $n=k+O(\log k)$, one has $F(n)-F(k)=o(1)$ and hence $\lceil F(n)\rceil\leq\lceil F(k)\rceil+1$ for all sufficiently large $k$.
Theorem~\ref{thm:random-explicit} yields a tournament of order $n$ with $\adn(T)\geq n-\lceil F(n)\rceil+1\geq k,$ which witnesses the stated upper bound.
\end{proof}

Following~\cite{BJPY}, let $\gamma$ be the infimum of the real numbers $c$ for which infinitely many tournaments $T$ satisfy $\adn(T)\geq |V(T)|-c\log_2|V(T)|$.

\begin{corollary}\label{cor:gamma}
The extremal constant satisfies
$$
 \frac1{\log_2(4/3)}\leq\gamma\leq\frac4{\log_2(8/7)}.
$$
\end{corollary}

\begin{proof}
The lower bound follows from \eqref{eq:BJPY-universal}; the upper bound follows from Theorem~\ref{thm:random-explicit}.
\end{proof}

\subsection{The first critical orders}
The values $m(1)=1$ and $m(2)=3$ are immediate.
We now determine the next two values and classify the tournament witnesses at their minimum orders.

Let $\Tstar$ be the tournament with vertex set $\{h,\ell,m_0,m_1,m_2\}$.
The set $M=\{m_0,m_1,m_2\}$ induces the directed triangle $m_0\to m_1\to m_2\to m_0$, and $\{h\}\Rightarrow M\Rightarrow\{\ell\}$, $\ell\to h.$

Let $P_7$ be the Paley tournament on $\mathbb Z_7$, with $x\to y$ whenever $y-x\in\{1,2,4\}$ modulo seven.
To define $Q_7$, take a transitive set $X=\{x_1,x_2,x_3,x_4\}$ in the displayed order.
For $y\notin X$, set $S_X(y)=\{i\in[4]:y\to x_i\}.$
Add three vertices $u,v,w$ with $S_X(u)=\{2\}, S_X(w)=\{1,3\}, S_X(v)=\{1,4\},$ and orient their mutual arcs as $w\to u\to v$ and $w\to v$.

\begin{figure}[t]
\centering
\includegraphics{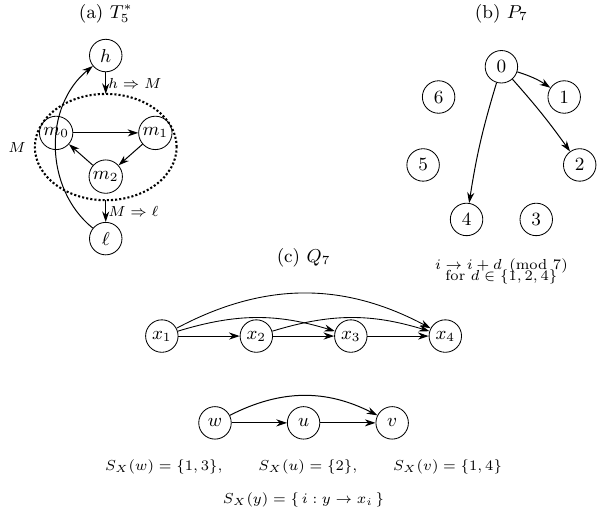}
\caption{Structural models of the critical tournaments.
Cluster arrows in~\textnormal{(a)} represent all arcs in the indicated direction.
In~\textnormal{(b)}, translation modulo seven supplies all translates of the three displayed generators.
In~\textnormal{(c)}, the pattern convention specifies every omitted cross-arc.}
\label{fig:critical-models}
\end{figure}

The three minimum-order tournament witnesses are represented in Figure~\ref{fig:critical-models}.
The pattern panel for $Q_7$ is also the form used in the comparison argument below.

The following observation is the central device in the transitive-four-set case.

\begin{lemma}[Bipartite comparison]\label{lem:comparison}
Let $T$ be a tournament and let $A,B$ be disjoint vertex sets.
Then $T[A,B]$ is acyclic if and only if the sets $N^+(a)\cap B$, $a\in A,$ are linearly ordered by inclusion.
In particular, in the notation above, $T[X,\{y,z\}]$ is acyclic if and only if $S_X(y)$ and $S_X(z)$ are comparable by inclusion.
\end{lemma}

\begin{proof}
Two incomparable neighbourhoods give an alternating directed $4$-cycle, and every alternating directed $4$-cycle gives two incomparable neighbourhoods.
The result follows from Lemma~\ref{lem:local}.
\end{proof}

\begin{theorem}\label{thm:critical-small}
For the oriented-graph parameter in \eqref{eq:mk-definition},
$$
 m(3)=5,\qquad m(4)=7.
$$
Among tournaments of order five, $\Tstar$ is, up to isomorphism, the unique tournament with acyclic dichromatic number at least three.
Among tournaments of order seven, the only tournaments with acyclic dichromatic number at least four are $P_7$ and $Q_7$.
Moreover, $\adn(\Tstar)=3$, $\adn(P_7)=\adn(Q_7)=4.$
\end{theorem}

\begin{proof}
By Proposition~\ref{prop:tournament-reduction}, it is enough to work with tournaments.

Every tournament on four vertices contains a transitive triple, and hence admits an acyclic $2$-dicolouring.
Thus $m(3)\geq5$.
Now let $T$ be a tournament on five vertices with $\adn(T)>2$.
Then $T$ contains no transitive subtournament of order four.
The possible score sequences are $(2,2,2,2,2)$, $(1,2,2,2,3)$, and $(1,1,2,3,3)$.
The first case admits an acyclic $2$-dicolouring, the third is impossible, while the second yields, up to isomorphism, the unique obstruction $\Tstar$; the score-sequence verification is given in Appendix~A.

The fact that $T_5^*$ admits no acyclic $2$-dicolouring already follows from the lightness argument of Bang-Jensen,Picasarri-Arrieta, and Yeo\cite[Section~4.2, proof of Theorem~21]{BJPY}.
Indeed, in their argument, the arc $uv$ corresponds to $\ell \to h$, and the directed triangle in $N^+(v) \cap N^-(u)$ corresponds to $M$.
We include the short proof for completeness.

Suppose, for a contradiction, that $T_5^*$ admits an acyclic $2$-dicolouring.
Write $M=\{a,b,c\}$, where $a \to b \to c \to a$, and recall that $\{h\} \Rightarrow M \Rightarrow \{\ell\}$ and $\ell \to h$.
Since $M$ is a directed triangle, it must use both colours.
By cyclic symmetry and an interchange of the colours, we may assume that $a$ and $b$ are red and $c$ is blue.
The vertices $h$ and $\ell$ must have different colours: otherwise, choosing a vertex $m \in M$ of their common colour would give a monochromatic directed triangle $\ell \to h \to m \to \ell$.
If $h$ is blue, then $\ell \to h \to b \to c \to \ell$ is an alternating directed $4$-cycle.
If $h$ is red, then $c \to a \to \ell \to h \to c$ is an alternating directed $4$-cycle.
Both cases contradict the definition of an acyclic dicolouring.
Thus $\vec{\chi}_a(T_5^*) \geq 3$.
Conversely, a transitive triple together with two singleton classes gives an acyclic $3$-dicolouring of $T_5^*$.
Therefore $\vec{\chi}_a(T_5^*)=3$, and, together with the four-vertex bound above, this yields $m(3)=5$.

We next prove $m(4)=7$.
Every tournament on at most six vertices admits an acyclic $3$-dicolouring; the verification is given in Appendix~A.
Hence $m(4)\geq7$.

Let $T$ be a tournament of order seven with $\adn(T)>3$, and let $t(T)$ denote the maximum order of a transitive subtournament of $T$.
Since every tournament contains a transitive triple, while a transitive set of order at least five together with the remaining vertices can be split into three acyclic colour classes, we have $t(T)\in\{3,4\}$.

Suppose first that $t(T)=3$.
Fix $v\in V(T)$ and put $A=\Nplus(v)$ and $B=\Nminus(v)$.
Neither $A$ nor $B$ contains a transitive triple; hence both have order three and induce directed triangles.
Thus $T$ is regular of outdegree three.
Write $a_0\to a_1\to a_2\to a_0$ for the cycle on $A$.
Regularity shows that the arcs from $B$ to $A$ form a perfect matching; after relabelling $B=\{b_0,b_1,b_2\}$, we may assume that $b_i\to a_i$ for each $i$, while $a_i\to b_j$ whenever $i\neq j$.

The directed triangle on $B$ must have the opposite cyclic orientation to that on $A$.
Indeed, if $b_0\to b_1\to b_2\to b_0$, then $\{a_0,b_1,a_1,b_2\}$ is transitive in the order $a_0,b_1,a_1,b_2$.
These relations determine all arcs of $T$.
Under the relabeling
$$
v\mapsto0,\qquad
(a_0,a_1,a_2)\mapsto(1,2,4),\qquad
(b_0,b_1,b_2)\mapsto(6,5,3),
$$
we have $x\to y$ precisely when $y-x\in\{1,2,4\}\pmod 7$.
Hence $T\cong P_7$.

Suppose now that $t(T)=4$.
Fix a transitive four-set $X=\{x_1,x_2,x_3,x_4\}$ in this order, and put $Y=V(T)\setminus X$.
For $y\in Y$, recall that $S_X(y)=\{i\in[4]:y\to x_i\}$.
None of the three patterns $S_X(y)$ can be $\varnothing$, $\{4\}$, $\{3,4\}$, $\{2,3,4\}$, or $[4]$, since in each of these cases $y$ can be inserted into the transitive order of $X$ to produce a transitive five-set.

Moreover, the three outside patterns must form an antichain under inclusion.
Indeed, if $S_X(y)$ and $S_X(z)$ were comparable for two vertices $y,z\in Y$, then Lemma~\ref{lem:comparison} would imply that $T[X,\{y,z\}]$ is acyclic.
The classes $X$, $\{y,z\}$, and $Y\setminus\{y,z\}$ would then form an acyclic $3$-dicolouring, a contradiction.

The remaining antichains are analysed in Appendix~A.
Up to taking the converse tournament and then reversing the transitive order of $X$, and up to relabelling $Y$, all nonexceptional patterns admit an acyclic $3$-dicolouring.
The unique exceptional pattern is $\{2\},\{1,3\},\{1,4\}$, and the only orientation of the three outside vertices that remains an obstruction gives precisely $Q_7$.
Hence every seven-vertex tournament with acyclic dichromatic number at least four is isomorphic to $P_7$ or $Q_7$.

Finally, the directed-triangle and alternating-$4$-cycle verifications in Appendix~A show that neither $P_7$ nor $Q_7$ admits an acyclic $3$-dicolouring, while the same appendix gives explicit acyclic $4$-dicolourings of both tournaments.
Thus $\adn(P_7)=\adn(Q_7)=4$, and consequently $m(4)=7$.
\end{proof}
\section*{Acknowledgments}
This work was supported by the Fundamental and Interdisciplinary Disciplines Breakthrough Plan of the Ministry of Education of China (JYB2025XDXM207).

\section*{Declaration of AI-assisted technologies}

During the preparation of this manuscript, the authors used ChatGPT to assist with creating figures, organizing the manuscript, and polishing the language.
ChatGPT also assisted in generating Python code used to verify selected graph examples.

\appendix
\section{Finite verifications for the critical-order theorem}
\label{app:critical-proof}

This appendix contains the finite verifications omitted from the proof of Theorem~5.6.
We give the five-vertex score-sequence analysis, the six-vertex bound, the remaining pattern analysis in the case of a transitive four-set, and the verification that $P_7$ and $Q_7$ have acyclic dichromatic number four.

\subsection{The five-vertex score-sequence analysis}

Let $T$ be a tournament on five vertices with no transitive subtournament of order four.
No vertex has outdegree $0$ or $4$: among the other four vertices there is a transitive triple, and adjoining a source or sink would produce a transitive four-set.
Since the sum of the outdegrees is $10$, the possible score sequences are $ (2,2,2,2,2)$, $(1,2,2,2,3)$, $(1,1,2,3,3)$.

The last sequence is impossible.
Suppose that $x\to y$ and both $x$ and $y$ have outdegree $3$.
Since $y$ is dominated by $x$, it must dominate each of the other three vertices.
The out-neighbourhood of $x$ consists of $y$ together with two of these three vertices, and these four vertices therefore induce a transitive tournament, a contradiction.

The regular tournament of order five is unique up to isomorphism.
In its cyclic model, where $i$ dominates $i+1$ and $i+2$ modulo $5$, the classes $\{0,1\}$ and $\{2,3,4\}$ form an acyclic $2$-dicolouring.
Indeed, the first class and the second class are transitive, and the bipartite subdigraph between them has topological order $4,1,3,0,2$.

It remains to consider the score sequence $(1,2,2,2,3)$.
Let $h$ and $\ell$ have outdegrees $3$ and $1$, respectively, and put $M=V(T)\setminus\{h,\ell\}$.

Suppose first that $\ell\to h$.
Degree considerations give $h\Rightarrow M\Rightarrow\{\ell\}$.
The subtournament $T[M]$ cannot be transitive, since adjoining $h$ would then give a transitive four-set.
Hence $T[M]$ is a directed triangle, and the resulting tournament is precisely $\Tstar$.

Suppose instead that $h\to\ell$.
The unique out-neighbour of $\ell$ in $M$ must be one of the two out-neighbours of $h$ in $M$; otherwise those two vertices together with $h$ and $\ell$ form a transitive four-set.
Label $M=\{a,b,c\}$ so that $h\to a,b$, $c\to h$, $\ell\to a$, and $b,c\to\ell$.
The absence of a transitive four-set forces $a\to b$ and $b\to c$, and the degree sequence then forces $a\to c$.
Thus $T[M]$ is transitive.
The classes $\{h,\ell\}$ and $M$ form an acyclic $2$-dicolouring: the bipartite subdigraph between them has topological order $c,h,b,\ell,a$.

This proves the score-sequence assertion used in the proof of Theorem~5.6.

\subsection{The six-vertex bound}

Let $T$ be a tournament on six vertices.
If $T$ contains a transitive four-set, that set together with two singleton classes gives an acyclic $3$-dicolouring.
We may therefore assume that $T$ contains no transitive four-set.

For every vertex $v$, neither $\Nplus(v)$ nor $\Nminus(v)$ contains a transitive triple, since adjoining $v$ would then produce a transitive four-set.
Hence both neighbourhoods have order at most three.
Since their orders sum to five, every vertex has outdegree $2$ or $3$, and the score sequence is $(2,2,2,3,3,3)$.

Let $A=\{a_0,a_1,a_2\}$ be the set of vertices of outdegree $2$ and let $B=\{b_0,b_1,b_2\}$ be the set of vertices of outdegree $3$.
We claim that neither $T[A]$ nor $T[B]$ is transitive.

Suppose, for a contradiction, that $T[A]$ is transitive, say $a_0\to a_1\to a_2$.
Since every vertex of $A$ has total outdegree $2$, the numbers of its out-neighbours in $B$ are respectively $0,1,2$.
Thus every vertex of $B$ dominates $a_0$, and exactly two vertices of $B$ dominate $a_1$.
No vertex of $B$ can dominate all of $A$, since that would produce a transitive four-set.
Hence the two vertices of $B$ dominating $a_1$, together with $a_0$ and $a_1$, form a transitive four-set, a contradiction.
The assertion for $B$ follows by taking the converse tournament.

Thus both $T[A]$ and $T[B]$ are directed triangles.
Every vertex of $A$ has exactly one out-neighbour in $B$, while every vertex of $B$ receives exactly one arc from $A$.
The arcs from $A$ to $B$ therefore form a perfect matching.
After relabelling $B$, we may assume that $a_i\to b_i$ for $i=0,1,2$, while $b_j\to a_i$ whenever $i\neq j$.

Now colour each pair $\{a_i,b_i\}$ with colour $i+1$.
Each colour class is acyclic.
For two distinct indices $i,j$, the bipartite digraph between $\{a_i,b_i\}$ and $\{a_j,b_j\}$ contains no directed $4$-cycle: the two cross-arcs $b_j\to a_i$ and $b_i\to a_j$ prevent an alternating cyclic orientation.
By Lemma~2.2, this is an acyclic $3$-dicolouring.

Consequently every tournament on at most six vertices is acyclic $3$-dicolourable.

\subsection{The remaining pattern analysis}

We continue with the notation used in the proof of Theorem~5.6.
Thus $X=\{x_1,x_2,x_3,x_4\}$ is transitive in this order, $Y=\{u,v,w\}$, and $S_X(y)=\{i\in[4]:y\to x_i\}.$
The main proof shows that the three outside patterns form an antichain and avoid the five insertion patterns $\varnothing$, $\{4\}$, $\{3,4\}$, $\{2,3,4\}$, and $[4]$.

For brevity, write $ij$ for $\{i,j\}$ and similarly for larger subsets.
We also write, for example, $13v\mid24w\mid u$ for the three colour classes $\{x_1,x_3,v\}$, $\{x_2,x_4,w\}$, and $\{u\}$.
In each row below, the vertices $u,v,w$ carry the three listed patterns in that order.

Taking the converse tournament and reversing the order of $X$ induces the involution $S\longmapsto\{j\in[4]:5-j\notin S\}.$
There are eleven admissible patterns, $1,2,3,12,13,14,23,24,123,124,134.$
Classifying three-element antichains by the sizes of their members gives respectively $1,2,5,10,5,2,1$ antichains of types $(1,1,1),(1,1,2),(1,2,2),(2,2,2), (2,2,3),(2,3,3),(3,3,3).$
Thus there are $26$ antichains.
After relabelling $Y$ and applying the above involution, they form $15$ orbits.
The table treats $14$ of them; the remaining orbit is represented by $\{2\},\{1,3\},\{1,4\}$.
Up to this involution and a permutation of $Y$, all nonexceptional antichains are covered by the following table.

{\small
$$
\begin{array}{c|l}
(S_X(u),S_X(v),S_X(w))
 & \text{acyclic $3$-dicolouring}\\ \hline

1,2,3
 &12\mid34u\mid vw
   \quad\text{or}\quad
   12\mid34v\mid uw\\

1,3,24
 &13v\mid24w\mid u\\

2,3,14
 &14w\mid23v\mid u\\ \hline

12,13,14
 &123\mid4u\mid vw
   \quad\text{or}\quad
   123\mid4v\mid uw\\

12,13,24
 &13v\mid24w\mid u\\

12,14,23
 &14v\mid23w\mid u\\

12,14,24
 &123\mid4v\mid uw
   \quad\text{or}\quad
   123\mid4w\mid uv\\

13,14,23
 &14v\mid23w\mid u\\

13,14,24
 &13u\mid24w\mid v\\

14,23,24
 &13w\mid2v\mid4u\\ \hline

1,23,24
 &13w\mid4u\mid2v\\

3,12,14
 &124u\mid v\mid3w\\

3,12,24
 &124u\mid v\mid3w\\

3,14,24
 &124u\mid3v\mid w
\end{array}
$$
}

In each row the displayed colour classes are transitive, and the interactions between nonsingleton classes are acyclic by Lemma~5.5.
In the three rows with two displayed colourings, choose the first colouring when respectively \(u\to v\), \(u\to v\), and \(w\to v\), and choose the second when the corresponding arc is reversed.

All single-colouring rows work for every orientation of $T[Y]$.

The only antichain not covered above is, up to the stated symmetries, $S_X(u)=\{2\},\qquad S_X(w)=\{1,3\},\qquad S_X(v)=\{1,4\}.$
The colouring $123\mid4u\mid vw$ is acyclic unless $w\to u\to v$; in the exceptional orientation the relevant alternating cycle is $x_4\to w\to u\to v\to x_4$.
If $v\to w$, however, the colouring $123\mid4vw\mid u$ is acyclic, since $\{x_4,v,w\}$ is transitive in the order $v\to x_4\to w$.
Therefore an obstruction must satisfy $w\to u\to v$ and $w\to v$.
Together with the three outside patterns, these arcs determine precisely $Q_7$.

\subsection{The tournaments \texorpdfstring{$P_7$}{P7} and \texorpdfstring{$Q_7$}{Q7}}

We first verify the value for $P_7$.
Every out-neighbourhood in $P_7$ is a directed triangle, so $t(P_7)=3$.
The automorphisms $x\mapsto ax+b$, where $a\in\{1,2,4\}$ and arithmetic is modulo $7$, act transitively on the transitive triples of $P_7$.

Indeed, let $p\to q\to r$ with $p\to r$ be a transitive triple.
Translation by $-p$, followed by multiplication by the inverse of $q-p$, sends $p$ to $0$ and $q$ to $1$.
The image $r'$ satisfies $r'\in\{1,2,4\}$ and $r'-1\in\{1,2,4\}$, and therefore $r'=2$.
Thus every transitive triple may be normalised to $C=\{0,1,2\}$.

If $P_7$ admitted an acyclic $3$-dicolouring, then, since $t(P_7)=3$, its colour-class sizes would have type $(3,2,2)$ or $(3,3,1)$.
After normalising a class of size three to $C=\{0,1,2\}$, the possibilities on the remaining four vertices are excluded as follows.

{\small
$$
\begin{array}{c|c}
\text{remaining non-singleton classes}
 & \text{alternating directed cycle}\\ \hline
34\mid56
 &0\to4\to1\to3\to0\\
35\mid46
 &3\to4\to5\to6\to3\\
36\mid45
 &0\to4\to1\to5\to0\\
345\mid6\ \text{or}\ 456\mid3
 &0\to4\to1\to5\to0
\end{array}
$$
}

Here $345$ and $456$ are the only transitive triples in $\{3,4,5,6\}$.
Hence $P_7$ has no acyclic $3$-dicolouring.
On the other hand, $\{0,1,2\}\mid\{3,5\}\mid\{4\}\mid\{6\}$ is an acyclic $4$-dicolouring.
The only interaction between two nonsingleton classes has topological order $1,5,2,3,0$.
Consequently $\adn(P_7)=4$.

We now turn to $Q_7$.
It is convenient to use the following equivalent structural representation:
$$
 m=x_3,\qquad
 (a_0,a_1,a_2)=(x_4,u,v),\qquad
 (b_0,b_1,b_2)=(w,x_2,x_1).
$$
Put \(A=\{a_0,a_1,a_2\}\) and \(B=\{b_0,b_1,b_2\}\).
Then \(B\Rightarrow\{m\}\Rightarrow A\), both \(A\) and \(B\) are directed triangles, with \(a_0\to a_1\to a_2\to a_0\) and \(b_0\to b_2\to b_1\to b_0\), and the cross-arcs satisfy \(a_i\to b_i\) and \(b_j\to a_i\) for \(i\neq j\).

Indices are read modulo $3$.
Simultaneously replacing every $a_i,b_i$ by $a_{i+1},b_{i+1}$ is an automorphism, and interchanging $a_i$ with $b_i$ gives an isomorphism between $Q_7$ and its converse.

The set $X=\{x_1,x_2,x_3,x_4\}$ is transitive, so $t(Q_7)\geq4$.
There is no transitive five-set.
Such a set would have to contain $m$ and two vertices from each of $A$ and $B$, since neither directed triangle contributes all three of its vertices.
Because $B\Rightarrow\{m\}\Rightarrow A$, the two selected vertices of $B$ would have to dominate both selected vertices of $A$.
By the cross-arc rule this would require two disjoint two-element subsets of $\{0,1,2\}$, which is impossible.
Hence $t(Q_7)=4$.

Suppose that $Q_7$ has an acyclic $3$-dicolouring.
Since no colour class has more than four vertices, its class-size type must be $(4,2,1)$, $(3,3,1)$, or $(3,2,2)$.

For a transitive class $C$ and a pair $R\subseteq V(Q_7)\setminus C$, call $R$ \emph{$C$-good} if $Q_7[C,R]$ is acyclic.

We first exclude type $(4,2,1)$.
Every transitive four-set contains $m$: otherwise it has two vertices from each of $A$ and $B$, and the cross-arc rule gives a directed triangle.
If it contains $m$, then the same rule forces either $\{m,a_i,a_{i+1},b_{i+2}\} \quad\text{or}\quad \{m,a_i,b_{i+1},b_{i+2}\}.$
Hence, up to cyclic shift, the transitive four-sets are precisely $\{m,a_0,a_1,b_2\} \quad\text{and}\quad \{m,a_0,b_1,b_2\}.$
The two families are interchanged by self-duality, so it is enough to consider $C=\{b_2,b_1,m,a_0\},$ in this transitive order.
The patterns of the three complementary vertices $a_1,a_2,b_0$ with respect to $C$ are respectively $\{2\}$, $\{1,4\}$, and $\{1,3\}$.
They are pairwise incomparable.
By Lemma~5.5, no pair of complementary vertices is $C$-good.
Hence type $(4,2,1)$ is impossible.

We next exclude type $(3,2,2)$.
Up to cyclic symmetry and self-duality, every transitive triple is represented by one of the four rows below.
Indeed, classify a triple according to whether it contains $m$ and according to its numbers of vertices in $A$ and $B$.
The cross-arc rule leaves one orbit of type $mAB$, one orbit of type $mAA$, and two orbits of type $AAB$; the remaining types follow by self-duality.
The second column lists all $C$-good pairs in its complement.

{\small
$$
\begin{array}{c|l}
C & C\text{-good pairs}\\ \hline
\{m,a_0,b_1\}
 &a_1b_2,\ a_2b_2,\ b_0b_2\\
\{a_0,a_1,b_2\}
 &mb_1,\ a_2b_1\\
\{m,a_0,a_1\}
 &a_2b_1,\ a_2b_2,\ b_0b_2,\ b_1b_2\\
\{a_0,a_1,b_0\}
 &ma_2,\ mb_2,\ a_2b_1,\ a_2b_2
\end{array}
$$
}

In the first two rows every good pair contains $b_2$ and $b_1$, respectively, so two disjoint good pairs do not exist.
In the third row the only good splitting is $a_2b_1\mid b_0b_2$, but $a_2\to b_2\to b_1\to b_0\to a_2$ is an alternating directed $4$-cycle.
In the fourth row the only good splitting is $a_2b_1\mid mb_2$, and $a_2\to b_2\to b_1\to m\to a_2$ is alternating.
Hence type $(3,2,2)$ is impossible.

It remains to exclude type $(3,3,1)$.
Equivalently, we must show that deleting any vertex from $Q_7$ does not leave an acyclic $2$-dicolouring.
By cyclic symmetry and self-duality, it is enough to delete either $m$ or $a_0$.

First delete $m$.
Since $A$ is a directed triangle, two consecutive vertices of $A$ have the same colour.
After a cyclic shift, suppose that $a_i,a_{i+1}$ are red and $a_{i+2}$ is blue.
The directed triangles
$$
 a_i\to a_{i+1}\to b_{i+1}\to a_i,\qquad
 a_{i+2}\to b_{i+2}\to b_{i+1}\to a_{i+2},
$$
and $a_i\to b_i\to b_{i+2}\to a_i$ force $b_{i+1}$ and $b_i$ to be blue and $b_{i+2}$ to be red.
But then $a_i\to b_i\to a_{i+1}\to a_{i+2}\to a_i$ is an alternating directed $4$-cycle.

Now delete $a_0$.
The directed triangle $B$ remains.
Call the colour appearing twice on $B$ red and the other colour blue.
The three possible repeated pairs give the following contradictions.

{\small
$$
\begin{array}{c|l}
\text{red pair in }B & \text{forced contradiction}\\ \hline

b_0,b_1
&a_1\text{ blue},\ a_2\text{ red};\quad
 \begin{cases}
 a_1\to a_2\to b_2\to m\to a_1
   \text{ alternating},&m\text{ red},\\
 a_2\to b_2\to b_1\to m\to a_2
   \text{ alternating},&m\text{ blue};
 \end{cases}\\[2mm]

b_1,b_2
&a_2\text{ blue};\quad
 \begin{cases}
 a_1\to b_1\to a_2\to b_2\to a_1
   \text{ alternating},&a_1\text{ blue},\\
 a_1\to a_2\to b_2\to m\to a_1
   \text{ alternating},&a_1\text{ red};
 \end{cases}\\[2mm]

b_2,b_0
&\begin{cases}
 a_2\to b_2\to b_1\to m\to a_2
   \text{ alternating},
   &a_2\text{ blue},\ m\text{ red},\\
 a_1\to a_2\to b_2\to m\to a_1
   \text{ alternating},
   &a_2\text{ blue},\ m\text{ blue},\\
 a_1\to a_2\to b_2\to a_1
   \text{ monochromatic},
   &a_2\text{ red}.
 \end{cases}
\end{array}
$$
}

In the second row, if $a_1$ is red then the directed triangle $m\to a_1\to b_1\to m$ forces $m$ to be blue.
In the final row, the colours required in the displayed alternatives follow from the same directed-triangle forcing.
Thus every vertex deletion fails to admit an acyclic $2$-dicolouring, and type $(3,3,1)$ is impossible.

We have proved that $Q_7$ has no acyclic $3$-dicolouring.
Conversely, $\{x_2,x_3,v,x_4\}\mid\{u\}\mid\{x_1\}\mid\{w\}$ is an acyclic $4$-dicolouring.
Its only nonsingleton colour class is transitive in the order $x_2\to x_3\to v\to x_4$, and every bipartite subdigraph involving a singleton class is acyclic.
Hence $\adn(Q_7)=4$.

This completes the finite verifications used in Theorem~5.6.

\bibliographystyle{plain}
\bibliography{ref}

\end{document}